\documentclass[12pt]{amsart}
\usepackage[all]{xy}

\usepackage{amssymb}

\newcommand{\Cdb}{\mbox{$\mathbb{C}$}}

\newcommand{\Fdb}{\mbox{$\mathbb{F}$}}

\newcommand{\Rdb}{\mbox{$\mathbb{R}$}}

\newcommand{\Tdb}{\mbox{$\mathbb{T}$}}

\newcommand{\Zdb}{\mbox{$\mathbb{Z}$}}

\newcommand{\M}{\mbox{${\mathcal M}$}}

\renewcommand{\P}{\mbox{${\mathcal P}$}}

\newcommand{\Rad}[1]{{\rm Rad}(#1)}

\newcommand{\norm}[1]{\Vert#1\Vert}

\newcommand{\bignorm}[1]{\bigl\Vert#1\bigr\Vert}

\newcommand{\Bignorm}[1]{\Bigl\Vert#1\Bigr\Vert}

\newcommand{\ten}{\overline{\otimes}}

\newtheorem{theorem}{Theorem}[section]\newtheorem{lemma}[theorem]{Lemma}
\newtheorem{corollary}[theorem]{Corollary}
\newtheorem{proposition}[theorem]{Proposition}
\theoremstyle{remark}
\newtheorem{remark}[theorem]{\bf Remark}\theoremstyle{definition}
\numberwithin{equation}{section}

\begin{document}

\title[]{Rademacher averages on noncommutative symmetric spaces}

\author{Christian Le Merdy, Fedor Sukochev}
\address{Laboratoire de Math\'ematiques\\ Universit\'e de  Franche-Comt\'e
\\ 25030 Besan\c con Cedex\\ France}
\email{clemerdy@univ-fcomte.fr}
\address{School of Informatics and Engineering\\ Flinders
University\\ Bedford Park\\ SA 5042, Australia}
\email{sukochev@infoeng.flinders.edu.au}

\date{\today}

\thanks{The first author is supported by the research program ANR-06-BLAN-0015,
the second author is supported by the ARC}

\begin{abstract} Let $E$ be a separable (or the dual of a separable)
symmetric function space, let $M$ be a semifinite von Neumann
algebra and let $E(M)$ be the associated noncommutative function
space. Let $(\varepsilon_k)_{k\geq 1}$ be a Rademacher sequence,
on some probability space $\Omega$. For finite sequences
$(x_k)_{k\geq 1}$ of $E(M)$, we consider the Rademacher averages
$\sum_k \varepsilon_k\otimes x_k$ as elements of the
noncommutative function space $E(L^\infty(\Omega)\ten M)$ and
study estimates for their norms $\Vert \sum_k \varepsilon_k\otimes
x_k\Vert_E$ calculated in that space. We establish general
Khintchine type inequalities in this context. Then we show that if
$E$ is 2-concave, $\Vert \sum_k \varepsilon_k\otimes x_k\Vert_E$
is equivalent to the infimum of $\norm{ (\sum y_k^*
y_k)^{\frac{1}{2}}} +  \norm{ (\sum z_k z_k^*)^{\frac{1}{2}}}$
over all $y_k,z_k$ in $E(M)$ such that $x_k=y_k+z_k$ for any
$k\geq 1$. Dual estimates are given when $E$ is 2-convex and has a
non trivial upper Boyd index. In this case, $\Vert \sum_k
\varepsilon_k\otimes x_k\Vert_E$ is equivalent to $\norm{ (\sum
x_k^* x_k)^{\frac{1}{2}}} +  \norm{ (\sum x_k
x_k^*)^{\frac{1}{2}}}$. We also study Rademacher averages
$\sum_{i,j} \varepsilon_i\otimes \varepsilon_j\otimes x_{ij}$ for
doubly indexed families $(x_{ij})_{i,j}$ of $E(M)$.
\end{abstract}

\maketitle

\noindent{\it Mathematics Subject Classification : Primary 46L52;
Secondary 46M35, 47L05}

\medskip\section{Introduction}

The purpose of this paper is to study Rademacher averages on
noncommutative symmetric spaces, in connection with some recent
developments of noncommutative Khintchine inequalities. Throughout
we let $E$ be a symmetric Banach function space on $(0,\infty)$
(see \cite{KPS}). Let $(M,\tau)$ be a semifinite von Neumann
algebra equipped with a normal semifinite faithful trace $\tau$.
For any $x$ belonging to the space $\widetilde{M}$ of all
$\tau$-measurable operators, let $\mu(x)\colon t>0\mapsto
\mu_t(x)$ denote the singular value function of $x$ (see
\cite{FK}). Then the noncommutative symmetric space $E(M)$ is the
space of all $x\in \widetilde{M}$ such that $\mu(x)\in E$,
equipped with the norm
$$
\norm{x}_{E(M)}\, =\, \norm{\mu(x)}_E.
$$
We refer to \cite{KPS, LT2} for general facts on symmetric
function spaces and to \cite{O, SC, DDD1, DDD2, DDD3, KS} for a
thorough study of $E(M)$-spaces and their properties. We note that
in the case when $E=L^p(0,\infty)$, the space $E(M)=L^p(M)$ is the
usual noncommutative $L^p$-space associated with $M$ (see e.g.
\cite{FK}, \cite {PX} and the references therein).

In general we will simply let $\norm{\ }_E$ (instead of $\norm{\
}_{E(M)}$) denote the norm on $E(M)$ if there is no risk of
confusion.

\bigskip
Let $(\Sigma, d\mu)$ be a localizable measure space and consider
the commutative von Neumann algebra $L^{\infty}(\Sigma)$. In the
sequel, we will always consider the von Neumann algebra tensor
product $L^{\infty}(\Sigma)\overline{\otimes} M$ as equipped with
the trace $d\mu\otimes\tau$. This gives rise to noncommutative
spaces $E(L^{\infty}(\Sigma)\ten M)$. Note that
$L^p(L^{\infty}(\Sigma)\ten M)$ coincides with the Bochner space
$L^p(\Sigma; L^p(M))$ for any $p\geq 1$. However
$E(L^{\infty}(\Sigma)\ten M)$ is not a Bochner $E(M)$-valued space
in general.

Consider the compact group $\Omega=\{-1,1\}^\infty$, equipped with
its normalized Haar measure $dm$. For $k\geq 1$, we let
$\varepsilon_k\colon\Omega\to \{-1,1\}$ be the Rademacher
functions defined by letting $\varepsilon_k(\Theta)=\theta_k$ for
any $\Theta=(\theta_k)_{k\geq 1}\in\Omega$. Let $x_1,\ldots, x_n$
be a finite family of $E(M)$. We  will consider two Rademacher
averages of the $x_k$'s. First we let
$$
\Bignorm{\sum_{k}\varepsilon_k \otimes
x_k}_{\Rad{E}}\,=\,\Bigl(\int_{\Omega}\Bignorm{\sum_{k}\varepsilon_k(\Theta)
x_k}_{E(M)} \,dm(\Theta)\,\Bigr)
$$
be the norm of the sum $\sum_k\varepsilon_k\otimes x_k\,$ in the
Bochner space $L^1(\Omega; E(M))$. Next we note that each
$\varepsilon_k\otimes x_k$ belongs to the noncommutative space
$E(L^{\infty}(\Omega)\ten M)$ and we let
$$
\Bignorm{\sum_{k}\varepsilon_k\otimes x_k}_{E}
$$
be the norm of their sum $\sum_k\varepsilon_k\otimes x_k\,$ in the
latter space.

Classical commutative or noncommutative Khintchine inequalities
involve the `classical' Rademacher averages expressed by $\norm{\
}_{\Rad{E}}$ (see e.g. \cite{LT2, LP, LPP, PX, LPX}). In this
paper we will be interested in Khintchine inequalities  regarding
the averages expressed by $\norm{\ }_E$. In general, the averages
$\norm{\ }_{\Rad{E}}$ and $\norm{\ }_E$ are not equivalent, see
Section 4 for more on this topic.

We let $p_E$ and $q_E$ denote the Boyd indices of $E$ (see
\cite[Def. 2.b.1]{LT2}).

In the case when $M=L^{\infty}(\Sigma)$ is a commutative von
Neumann algebra, it follows from \cite[Prop. 2.d.1]{LT2} and its
proof that whenever $q_E<\infty$, we have an equivalence
$$
\Bignorm{\sum_k\varepsilon_k\otimes x_k}_{E}\,\approx\,
\Bignorm{\Bigl( \sum_k \vert
x_k\vert^2\Bigr)^{\frac{1}{2}}}_{E(\Sigma)}
$$
for finite families $(x_k)_{k}$ of $E(\Sigma)$. The main purpose
of this paper is to establish noncommutative versions of this
theorem.

To express them, we first note that whenever $A(\cdotp)$ and
$B(\cdotp)$ are two quantities depending on a parameter $\omega$,
we will write $A(\omega)\lesssim B(\omega)$ provided that there is
a positive constant $K$ such that $A(\omega)\leq K B(\omega)$ for
any $\omega$. Then we write $A(\omega)\approx  B(\omega)$ when we
both have $A(\omega)\lesssim B(\omega)$ and $A(\omega)\gtrsim
B(\omega)$.

Given an arbitrary $M$ and a finite family $(x_k)_{k}$ of $E(M)$,
we set
\begin{equation}\label{1Square}
\bignorm{(x_k)_k}_{c}\,=\,\Bignorm{\Bigl(\sum_{k}
x_k^*x_k\Bigr)^{\frac{1}{2}}}_{E(M)}\quad\hbox{and}\quad
\bignorm{(x_k)_k}_{r}\,=\,\Bignorm{\Bigl(\sum_{k} x_kx_k^*
\Bigr)^{\frac{1}{2}}}_{E(M)}.
\end{equation}
Next we let $$
\bignorm{(x_k)_k}_{\max}\,=\,\max\Bigl\{\bignorm{(x_k)_k}_{c}\,
,\, \bignorm{(x_k)_k}_{r}\Bigr\}
$$
and
$$
\bignorm{(x_k)_k}_{\inf}\,=\,\inf\Bigl\{\bignorm{(y_k)_k}_{c}\, +
\,\bignorm{(z_k)_k}_{ r}\Bigr\},
$$
where the infimum runs over all finite families $(y_k)_k$ and
$(z_k)_k$ in $E(M)$ such that $x_k=y_k+z_k$ for any $k\geq 1$.

The classical noncommutative Khintchine inequalities \cite{LP,
LPP, PX} say that if $2\leq p<\infty$, then we have an equivalence
\begin{equation}\label{1Khintchine1} \Bignorm{\sum_k
\varepsilon_k\otimes
x_k}_{\Rad{L^p}}\,\approx\,\bignorm{(x_k)_k}_{\max}
\end{equation}
for finite families $(x_k)_k$ of $L^p(M)$, whereas if $1\leq p\leq
2$, we have an equivalence
\begin{equation}\label{1Khintchine2}
\Bignorm{\sum_k \varepsilon_k\otimes
x_k}_{\Rad{L^p}}\,\approx\,\bignorm{(x_k)_k}_{\inf}.
\end{equation}
Furthermore it is shown in \cite{LPX} that (\ref{1Khintchine1})
remains true when $L^p$ is replaced by any $E$ which is $2$-convex
and $q$-concave for some $q<\infty$, and that (\ref{1Khintchine2})
remains true when $L^p$ is replaced by any $E$ which is
$2$-concave.

The main result of this paper is the following.

\begin{theorem}\label{1Main} Assume that $E$ is separable, or that $E$ is the dual
of a separable symmetric function space.
\begin{itemize}
\item [(1)] If $q_E<\infty$, then we have
\begin{equation}\label{1K1}
\bignorm{(x_k)_k}_{\inf}\,\lesssim\,\Bignorm{\sum_k\varepsilon_k\otimes
x_k}_{E}
\end{equation}
for finite families $(x_k)_k$ of $E(M)$. \item [(2)] If
$q_E<\infty$ and $p_E>1$, then we have
\begin{equation}\label{1K2}
\Bignorm{\sum_k\varepsilon_k\otimes x_k}_{E}\,\lesssim\,
\bignorm{(x_k)_k}_{\max}
\end{equation}
for finite families $(x_k)_k$ of $E(M)$.
\end{itemize}
\end{theorem}

\bigskip
This theorem will be proved in Section 3. The preceding Section 2
contains preparatory and preliminary results. Section 4 is devoted
to examples and  illustrations. In Corollary \ref{5Equiv3}, we
show that the Rademacher averages $\bignorm{\sum_k
\varepsilon_k\otimes x_k}_E$ are equivalent to
$\norm{(x_k)_k}_{\inf}$ if $E$ is 2-concave, and that they are
equivalent to $\norm{(x_k)_k}_{\max}$ if $E$ is 2-convex and
$q_E<\infty$. These are analogs of the main results of \cite{LPX}.
Also we discuss equivalence between $\bignorm{\sum_k
\varepsilon_k\otimes x_k}_E$ and the classical averages
$\bignorm{\sum_k \varepsilon_k\otimes x_k}_{\Rad{E}}$ and show
that Theorem \ref{1Main} is in some sense optimal. In Section 5,
we study double sums $$
\sum_{i,j}\varepsilon_i\otimes\varepsilon_j\otimes x_{ij},
$$
regarded as elements of $E(L^{\infty}(\Omega)\ten
L^{\infty}(\Omega)\ten M)$. We extend Theorem \ref{1Main} and
Corollary \ref{5Equiv3} to this setting.

We mention (as an open problem) that we do not know if the second
part of Theorem \ref{1Main} remains true without assuming that
$p_E>1$.

\medskip\section{Preliminaries and background}

In this section, we assume that $E$ is a fully symmetric function
space on $(0,\infty)$ (in the sense of \cite{DDD2}). We do not
make any assumption on its Boyd indices. We let $(M,\tau)$ and
$(N,\sigma)$ denote arbitrary semifinite von Neumann algebras.

Given $1\leq p\leq q\leq \infty$, we will write $$ E\in {\rm
Int}(L^p,L^q)
$$
provided that the following interpolation property holds: Whenever
$T$ is a linear operator from  $L^p(0,\infty) +  L^q(0,\infty)$
into itself which is bounded from $L^p(0,\infty)$ into
$L^p(0,\infty)$ and from $L^q(0,\infty)$ into $L^q(0,\infty)$,
then $T$ maps $E$ into itself. We recall that in this case, the
resulting operator
$$
T\colon E\longrightarrow E
$$
is automatically bounded. We refer e.g. to  \cite[I., Sect.
4]{KPS} or \cite{KM} for the necessary background on
interpolation. We recall that if $E\in {\rm Int}(L^p,L^q)$, then
we have
\begin{equation}\label{2Encad}
p\leq p_E\leq q_E\leq q.
\end{equation}
We also note that the fully symmetric assumption on $E$ is
equivalent to the property $E\in{\rm Int}(L^1,L^\infty)$ (see e.g.
\cite[II, Thm. 4.3]{KPS}). Further, any symmetric function space
which is either separable, or is the dual of a separable symmetric
function space, is automatically fully symmetric.

Throughout the paper, we will use the following fundamental result
from \cite{DDD2}.

\begin{proposition}\label{2Interpolation1}
Assume that $E\in {\rm Int}(L^p,L^q)$ and let
$$
T\colon L^p(M) + L^q(M) \longrightarrow L^p(N) + L^q(N)
$$
be any linear operator such that $$ T\colon L^p(M)\longrightarrow
L^p(N)\qquad\hbox{and}\qquad  T\colon L^q(M)\longrightarrow L^q(N)
$$
are bounded. Then $T$ maps $E(M)$ into $E(N)$ and the resulting
operator $T\colon E(M)\to E(N)$  is bounded. Moreover we have an
estimate
$$
\norm{T\colon E(M)\to E(N)}\leq C\,\max\bigl\{\norm{T\colon
L^p(M)\to L^p(N)}\, ,\, \norm{T\colon L^q(M)\to L^q(N)}\bigr\}
$$
for some constant $C$ not depending on either $M$, $N$, or $T$.
\end{proposition}

\bigskip
The norms introduced in (\ref{1Square}) are related to matrix
representations, as follows. Let $n\geq 1$ be an integer and
consider the von Neumann algebra $M_n(M)$ equipped with the trace
$tr\otimes \tau$ (here $tr$ is the usual trace on $M_n$). In the
sequel we will write $E_{ij}$ for the usual matrix units of $M_n$.
For any $x\in E(M)$, we have $$ \vert E_{ij}\otimes x\vert
=E_{jj}\otimes \vert x\vert,
$$
hence $E_{ij}\otimes x\in E(M_n(M))$, with $\norm{E_{ij}\otimes
x}_E=\norm{x}_E$. We deduce that the space $E(M_n(M))$ can be
algebraically identified with the space $M_n\otimes E(M)$ of
$n\times n$ matrices with entries in $E(M)$. Then for any
$x_1,\ldots,x_n$ in $E(M)$, we have
$$
\Bigl\vert \sum_{k=1}^{n} E_{k1}\otimes x_k\Bigr\vert =
E_{11}\otimes \Bigl(\sum_{k=1}^n x_k^* x_k\Bigr)^{\frac{1}{2}},
$$
hence
\begin{equation}\label{2Square}
\Bignorm{\Bigl(\sum_{k=1}^n x_k^*x_k\Bigr)^{\frac{1}{2}}}_{E(M)}\,
=\, \Bignorm{\sum_{k=1}^{n} E_{k1}\otimes x_k}_{E(M_n(M))}\,
=\,\left\Vert\left[\begin{array}{cccc} x_1 & 0 & \cdots &
0\\\vdots &\vdots & \ &\vdots\\   \vdots &\vdots & \ &\vdots\\x_n
& 0 & \cdots & 0\end{array}\right]\right\Vert_{E(M_n(M))}.
\end{equation}
Likewise, we have
$$
\Bignorm{\Bigl(\sum_{k=1}^n x_k
x_k^*\Bigr)^{\frac{1}{2}}}_{E(M)}\,=\, \Bignorm{\sum_{k=1}^n
E_{1k} \otimes
x_k}_{E(M_n(M))}\,=\,\left\Vert\left[\begin{array}{ccc} x_1 &
\cdots & x_n\\  0 &\cdots  & 0 \\  \vdots &  \ &\vdots \\0  &
\cdots & 0\end{array}\right] \right\Vert_{E(M_n(M))}.
$$

Let ${\rm Col}_n\colon M_n(M)\to M_n(M)$ be the natural projection
onto the `column subspace' of $M_n(M)$, taking $E_{i1}\otimes x$
to itself for any $i\geq 1$ and taking $E_{ij}\otimes x$ to $0$
whenever $j\geq 2$. Alternatively, ${\rm Col}_n$ is the right
multiplication $z\mapsto zc$, where $c=E_{11}\otimes 1$. Since
$\mu_t(zc)\leq \mu_t(z)\norm{c}_M=\mu_t(z)$ for any
$z\in\widetilde{M}$ and any $t>0$, the mapping ${\rm Col_n}$
extends to a contractive projection on $E(M_n(M))$, that is,
\begin{equation}\label{2Col}
\bignorm{{\rm Col}_n\colon E(M_n(M))\longrightarrow E(M_n(M))}\leq
1.
\end{equation}
We record for further use the following straightforward
consequence of the latter observation and Proposition
\ref{2Interpolation1}.

\begin{lemma}\label{2Interpolation2}
Assume that $E \in{\rm Int}(L^p, L^q)$. There is a constant $C$
verifying the following property. Let $n\geq 1$ be an integer and
let $T\colon L^p(M)^n + L^q(M)^n\to L^p(N)+L^q(N)$ be a linear map
such that for $r$ equal to either $p$ or $q$, we have
$$
\bignorm{T(x_1,\ldots, x_n)}_r\,\leq\, \Bignorm{\Bigl(\sum_{k=1}^n
x_k^* x_k\Bigr)^{\frac{1}{2}}}_r,\qquad x_1,\ldots, x_n\in L^r(M).
$$
Then we have  $$ \bignorm{T(x_1,\ldots, x_n)}_E\,\leq\, C
\Bignorm{\Bigl(\sum_{k=1}^n x_k^*
x_k\Bigr)^{\frac{1}{2}}}_E,\qquad x_1,\ldots, x_n\in E(M).
$$
\end{lemma}

Let $n\geq 1$ be an integer and let $\bigl(E(M)^n,\norm{\
}_{\inf}\bigr)$ (resp. $\bigl(E(M)^n,\norm{\ }_{\max}\bigr)$)
denote the product space $E(M)^n$ of all $n$-tuples $(x_1,\ldots,
x_n)$ of $E(M)$ equipped with the norm $\norm{(x_k)_k}_{\inf}$
(resp. $\norm{(x_k)_k}_{\max}$). Let $E'$ denote the K\"othe dual
of $E$ and recall that when $E$ is separable, then $E'=E^*$ (see
e.g. \cite[p. 102]{KPS}).

\begin{proposition}\label{2Duality}
Let $n\geq 1$ be an  integer.
\begin{itemize}
\item [(1)] If $E$ is separable, then we both have $$
\bigl(E(M)^n,\norm{\ }_{\inf}\bigr)^*\, =\, \bigl(E'(M)^n,\norm{\
}_{\max}\bigr) \quad\hbox{and}\quad \bigl(E(M)^n,\norm{\
}_{\max}\bigr)^*\, =\, \bigl(E'(M)^n,\norm{\ }_{\inf}\bigr)
$$
isometrically. \item [(2)] If $E$ is separable, or if $E$ is the
dual of a separable symmetric space, then for any $x_1,\ldots,
x_n$ in $E(M)$, we have
$$
\bignorm{(x_k)_k}_{\inf}\,=\,\sup\bigl\{\bignorm{(sx_ks)_k}_{\inf}\,
:\, s\ \hbox{is a selfadjoint projection from}\ M,\
\tau(s)<\infty\,\bigr\}.
$$
\end{itemize}
\end{proposition}

\begin{proof}
Assume that $E$ is separable. Then it follows from \cite[p.
745]{DDD3} that $E(M_n(M))^*=E'(M_n(M))$. Using (\ref{2Col}) and
its row counterpart, this implies that
$$
\bigl(E(M)^n,\norm{\ }_{c}\bigr)^*\, =\, \bigl(E'(M)^n,\norm{\
}_{r}\bigr) \quad\hbox{and}\quad \bigl(E(M)^n,\norm{\
}_{r}\bigr)^*\, =\, \bigl(E'(M)^n,\norm{\ }_{c}\bigr)
$$
isometrically. Then part (1) of the proposition follows at once,
using standard duality principles (see e.g. \cite[Thm.
I.3.1]{KPS}).

Now take $x_1,\ldots, x_n$ in $E(M)$. By the preceding point,
there exist $y_1,\ldots, y_n$ in $E'(M)$ such that
$\norm{(y_k)_k}_{\max}=1$ and
$$
\bignorm{(x_k)_k}_{\inf}\, =\, \sum_{k=1}^n \tau(x_k y_k).
$$
Let $\varepsilon >0$. For every $k=1,\ldots, n$, there exists a
selfadjoint projection $s_k$ from $M$ such that $\tau(s_k)<\infty$
and
$$
\bigl\vert \tau(sx_ksy_k) -
\tau(x_ky_k)\bigr\vert\,\leq\,\frac{\varepsilon}{n}
$$
for every $s\geq s_k$. Indeed, this follows from \cite[Prop. 2.5]
{CS} (the latter proposition holds true for general semifinite von
Neumann algebras). Set $s= \mathop{\bigvee}\limits_{1\leq k\leq n}
s_k$. Clearly $\tau(s)<\infty$ and we have
$$
\Bigl\vert \sum_{k=1}^n \tau(sx_ksy_k)\,\Bigr\vert\,\geq
\,\bignorm{(x_k)_k}_{\inf} -\varepsilon.
$$
Furthermore, $$ \Bigl\vert \sum_{k=1}^n
\tau(sx_ksy_k)\,\Bigr\vert\,\leq \,
\bignorm{(sx_ks)_k}_{\inf}\bignorm{(y_k)_k}_{\max} =
\bignorm{(sx_ks)_k}_{\inf}
$$
by part (1) of this proposition. Hence
$\norm{(sx_ks)_k}_{\inf}\geq \norm{(x_k)_k}_{\inf} -\varepsilon$.
This shows the non trivial inequality of part (2) in the case when
$E$ is separable. The proof in the case when $E$ is the dual of a
separable symmetric space is similar, by applying part (1) to the
predual of $E$.
\end{proof}

\begin{remark}\label{2Kconvex} Suppose that $E$ is separable.
Applying \cite[p. 745]{DDD3} as above, we have
\begin{equation}\label{2Dual}
E(L^{\infty}(\Omega)\ten M)^* = E'(L^{\infty}(\Omega)\ten M).
\end{equation}
Let $P\colon L^2(\Omega)\to L^2(\Omega)$ be the orthogonal
projection onto the closed linear span of the $\varepsilon_k$'s.
For any $1<r<\infty$, $L^r(M)$ is a $K$-convex Banach space in the
sense of \cite{Pi} (see also \cite{M2}). Indeed, $L^r(M)$ has a
non trivial type (see e.g. \cite[Cor. 5.5]{PX}). Thus  the linear
map $P\otimes I_{L^r(M)}$ extends to a bounded projection
$L^r(\Omega;L^r(M))\to L^r(\Omega;L^r(M))$.

Assume now that $p_E>1$ and $q_E<\infty$. Owing to the relations
$p_{E'}=\frac{q_E}{q_E-1}$ and $q_{E'}=\frac{p_E}{p_E-1}$
\cite[Prop. 2.b.2]{LT2}, we have $p_{E'}>1$ and $q_{E'}<\infty$ as
well.  Then Boyd's Theorem (see e.g. \cite[Thm. 7.3]{KM}) ensures
that $E\in {\rm Int}(L^p,L^q)$ and $E'\in {\rm Int}(L^p,L^q)$ for
some $1<p\leq q<\infty$. Applying the boundedness of $P\otimes
I_{L^r(M)}$ above with $r=p$ and $r=q$, we deduce by interpolation
that $P\otimes I_{E}$ and $P\otimes I_{E'}$ extend to bounded
projections
$$
E(L^{\infty}(\Omega)\ten M)\longrightarrow
E(L^{\infty}(\Omega)\ten M) \qquad\hbox{and}\qquad
E'(L^{\infty}(\Omega)\ten M) \longrightarrow
E'(L^{\infty}(\Omega)\ten M),
$$
respectively. Combining with the duality identification
(\ref{2Dual}), this yields equivalence properties
$$
\Bignorm{\sum_{k}\varepsilon_k \otimes
x_k}_E\,\approx\,\sup\Bigl\{\Bigl\vert \sum_k
\tau(x_ky_k)\Bigr\vert\, :\, y_k\in E',\
\Bignorm{\sum_{k}\varepsilon_k \otimes y_k}_{E'}\leq 1\,\Bigr\}
$$
and
$$
\Bignorm{\sum_{k}\varepsilon_k \otimes
y_k}_{E'}\,\approx\,\sup\Bigl\{\Bigl\vert \sum_k
\tau(x_ky_k)\Bigr\vert\, :\, x_k\in E,\
\Bignorm{\sum_{k}\varepsilon_k \otimes x_k}_{E}\leq 1\,\Bigr\}.
$$
We refer to \cite[Thm. 1.3]{DS} for a slightly more precise result
when $E$ is reflexive.
\end{remark}

\bigskip
Let $E^{(2)}$ denote the 2-convexification of $E$. We will use the
following well-known Cauchy-Schwarz inequality. For any $g,h\in
E^{(2)}(M)$, we have $gh\in E(M)$ and
\begin{equation}\label{2CS}
\norm{gh}_{E(M)}\leq \norm{g}_{E^{(2)}(M)}\norm{h}_{E^{(2)}(M)}.
\end{equation}
Indeed, this inequality follows from \cite[Thm 4.2]{FK}.

We will use Hardy spaces associated to symmetric function spaces.
We start with a general definition. Assume here that $(N,\sigma)$
is a finite von Neumann algebra and let $H^{\infty}(N)\subset N$
be a finite subdiagonal algebra in the sense of \cite[Section
8]{PX}. Recall that for any $1\leq p<\infty$, the associated Hardy
space $H^p(N)$ is defined as the closure of $H^{\infty}(N)$ into
$L^p(N)$ and that we actually have $H^p(N)=H^1(N)\cap L^p(N)$ (see
\cite{S} and \cite[(3.1)]{MW}). Note that since $N$ is finite, we
have a continuous inclusion $E(N)\subset L^1(N)$. Then we let
$$
H^{E}(N) =  H^1(N)\cap E(N),
$$
that we regard as a subspace of $E(N)$ equipped with the induced
norm. It is plain that $H^{E}(N) \subset E(N)$ is closed.

We clearly have $E^{(2)}(N)\subset L^2(N)$, and hence
$H^{E^{(2)}}(N)\subset H^2(N)$. Owing to the fact that the product
of two elements of $H^2(N)$ belongs to $H^1(N)$, we deduce that
for any $x,y\in H^{E^{(2)}}(N)$, the product $xy$ belongs to
$H^{E}(N)$. As a converse to this embedding and its companion
estimate (\ref{2CS}), we have the following factorization
property.

\begin{proposition}\label{2Hardy1}
For every $f$ in $H^{E}(N)$ and for every $\varepsilon >0$, there
exist $g,h\in H^{E^{(2)}}(N) $ such that $f=gh$ and
$$
\norm{g}_{E^{(2)}}
\norm{h}_{E^{(2)}}\leq(1+\varepsilon)\norm{f}_{E}.
$$
\end{proposition}

\begin{proof} In the case when $E=L^p$, this result is due to
Marsalli-West \cite[Thm. 4.3]{MW}. In turn, the latter result
relies on \cite[Thm. 4.2]{MW}. It is not hard to adapt the proof
of these two theorems to the above general case. An alternative
route consists in adapting the proof of \cite[Thm. 3.4]{BX}.
Details are left to the reader.
\end{proof}

\begin{remark}\label{2Hardy2}
When dealing with noncommutative Hardy spaces as above, it is
usually assumed that the trace $\sigma$ on $N$ is normalized, i.e.
$\sigma(1)=1$. However the proofs of \cite[Thms 4.2 and 4.3]{MW}
work as well under the mere assumption that $\sigma(1)<\infty$.
Likewise, the above Proposition \ref{2Hardy1} holds true whenever
$\sigma(1)<\infty$.
\end{remark}

In Section 3 we will apply the previous proposition in the
following classical context. We let $\Tdb$ denote the unit circle,
equipped with Haar measure, and we identify $L^{\infty}(\Tdb)$
with the space of essentially bounded $2\pi$-periodic functions
from $\Rdb$ into $\Cdb$ in the usual way. Next let $(M,\tau)$ be a
finite von Neumann algebra and let
$$
N=L^\infty(\Tdb)\ten M.
$$
Since $L^1(N)=L^1(\Tdb; L^1(M))$ one can define Fourier
coefficients on $L^1(N)$ by letting
$$
\widehat{f}(j)=\frac{1}{2\pi}\,\int_{0}^{2\pi} f(t)\,e^{-ijt}\,dt\
\in L^1(M),\qquad f\in L^1(N),\ j\in\Zdb.
$$
Then the space $$ H^{\infty}(N)\, =\, \bigl\{f\in N\, :\,
\forall\, j<0,\ \widehat{f}(j)=0\bigr\}
$$
is a finite subdiagonal algebra. Note that the resulting Hardy
spaces coincide with the vector valued $H^p$-spaces $H^p(\Tdb;
L^p(M))$, for any $1\leq p< \infty$. It is clear that
$$
H^{E}(L^\infty(\Tdb)\ten M)=\bigl\{f\in E(L^\infty(\Tdb)\ten M)\,
:\, \forall\, j<0,\ \widehat{f}(j)=0\bigr\}
$$

We conclude this section with a few remarks on Fourier
coefficients on $E(L^\infty(\Tdb)\ten M)$. For any $j \in\Zdb$,
let $F_j$ be the linear map taking any $f\in L^1(\Tdb; L^1(M))$ to
$\widehat{f}(j)$. Then $F_j$ is both a contraction from $L^1(\Tdb;
L^1 (M))$ into $L^1(M)$ and from $L^{\infty}(\Tdb)
\overline{\otimes} M$ into $M$. Hence by the interpolation
Proposition \ref{2Interpolation1}, it also extends to a bounded
operator
$$
F_{j,E}\colon E(L^\infty(\Tdb)\ten M)\longrightarrow E(M).
$$
Indeed there exists a constant $C_E$ (only depending on $E$) such
that
\begin{equation}\label{2Fourier}
\norm{\widehat{f}(j)}_{E(M)}\,\leq
C_E\,\norm{f}_{E(L^{\infty}(\footnotesize{\Tdb})\ten M)},\qquad
f\in E(L^\infty(\Tdb)\ten M), \ j\in\Zdb.
\end{equation}

\begin{lemma}
\label{2Ampli} For any $f_1,\ldots,f_n$ in $E(L^\infty(\Tdb)\ten
M)$, and for any $j\in\Zdb$, we have
$$
\Bignorm{\Bigl(\sum_{k=1}^{n}
\widehat{f_k}(j)^*\widehat{f_k}(j)\Bigr)^{\frac{1}{2}}}_{E(M)}\,
\leq\,C_E\,\Bignorm{\Bigl(\sum_{k=1}^{n} f_k^*
f_k\Bigr)^{\frac{1}{2}}}_{E(L^{\infty}(\footnotesize{\Tdb})\ten
M)}.
$$
\end{lemma}

\begin{proof}
We apply the above results on $M_n(M)$, with
$$
f\,=\,\sum_{k=1}^{n} E_{k1}\otimes f_k.
$$
For any $j$, the Fourier coefficient $\widehat{f}(j)$ is equal to
$\sum_{k} E_{k1}\otimes \widehat{f_k}(j)$. Hence the result
follows from (\ref{2Fourier}) and (\ref{2Square}).
\end{proof}

\medskip\section{Proof of the main result}

Throughout this section we let $(M,\tau)$ be a  semifinite von
Neumann algebra. Our aim is to prove Theorem \ref{1Main}. We start
with an equivalence property which will enable us to deduce the
estimation of the Rademacher averages on $E(M)$ from an estimation
of a certain lacunary Fourier series.

We set $e_j(t)=e^{ijt}$ for any integer $j\in\Zdb$ and any
$t\in\Rdb$. The left hand-side in the equivalence below is the
norm of $\sum_{k} e_{3^k}\otimes x_k$ in $E(L^{\infty}(\Tdb)\ten
M)$. Note that no assumption on either $p_E$ or $q_E$ is made in
the next statement.

\begin{proposition}\label{2Sidon1} Assume that $E$ is fully symmetric.
Then we have an equivalence
$$
\Bignorm{\sum_{k} e_{3^k}\otimes
x_k}_{E}\,\approx\,\Bignorm{\sum_{k}\varepsilon_k\otimes x_k}_{E}
$$
for finite families $(x_k)_k$ of $E(M)$.
\end{proposition}

\begin{proof} For any $k\geq 1$, let $\eta_k(t)=\cos(3^k t)$ and
$\widetilde{\eta}_k(t)=\sin(3^k t)$ for $t\in\Rdb$. It is plain
that
\begin{equation}\label{2cos}
\Bignorm{\sum_k \eta_k\otimes x_k}_E\leq
\Bignorm{\sum_ke_{3^k}\otimes x_k}_E\leq \Bignorm{\sum_k
\eta_k\otimes x_k}_E+\Bignorm{\sum_k \widetilde{\eta}_k\otimes
x_k}_E.
\end{equation}
Here the norms are computed in $E(L^{\infty}(\Tdb)\ten M)$.
Indeed, the mapping $g(t)\mapsto g(-t)$ is an isometry on the
latter space, which proves the first inequality. The second one is
obvious.

Now fix an integer $n\geq 1$ and consider the Riesz product
$$
K(\Theta,t)
=\prod_{k=1}^{n}\bigl(1+\varepsilon_k(\Theta)\eta_k(t)\bigr),\qquad\Theta\in\Omega,\
t\in\Rdb.
$$
Note that $K$ is a nonnegative function. For any
$A\subset\{1,\ldots,n\}$, set $\varepsilon_A=\prod_{k\in
A}\varepsilon_k$ and $\eta_A=\prod_{k\in A}\eta_k$. By convention,
$\varepsilon_{\emptyset}=1$ and $\eta_{\emptyset}=1$. If
$A\not=\emptyset$, then the integrals of $\varepsilon_A$ and
$\eta_A$ on $\Omega$ and $\Tdb$ respectively are equal to $0$.
Since
$$
K(\Theta,t)
=\,\sum_{A\subset\{1,\ldots,n\}}\varepsilon_A(\Theta)\,\eta_A(t),\qquad
\Theta\in\Omega, t\in\Rdb,
$$
this implies that
$$
\sup_\Theta\int_{0}^{2\pi} \vert
K(\Theta,t)\vert\,\frac{dt}{2\pi}\,=1 \qquad\hbox{and}\qquad
\sup_t\int_{\Omega}\vert K(\Theta,t)\vert\, dm(\Theta)\,=1.
$$
Consequently, one can define two linear contractions
$$
T_1\colon L^1(\Omega;L^1(M)) \longrightarrow L^1(\Tdb;L^1(M))
\qquad\hbox{and}\qquad   T_2\colon L^1(\Tdb;L^1(M))
\longrightarrow L^1(\Omega;L^1(M))
$$
by letting
$$
[T_1(f)](t)=\int_\Omega K(\Theta,t) f(\Theta) \, dm(\Theta),\qquad
f\in L^1(\Omega,L^1(M)),
$$
and
$$
[T_2(g)](\Theta)=\frac{1}{2\pi}\,\int_0^{2\pi} K(\Theta,t) g(t) \,
dt,\qquad g\in L^1(\Tdb,L^1(M)).
$$
Moreover $T_2^*\colon L^\infty(\Omega)\overline{\otimes} M\to
L^\infty(\Tdb)\overline{\otimes} M$ and $T_1$ coincide on the
intersection of their domains. Thus we may define a linear map
$$
T\colon L^1(\Omega;L^1(M))+
L^\infty(\Omega)\overline{\otimes}M\longrightarrow
L^1(\Tdb;L^1(M))+ L^\infty(\Tdb)\overline{\otimes}M
$$
extending both of them. Thus by interpolation (using Proposition
\ref{2Interpolation1}), there is a constant $C\geq 1$ (not
depending  on $n$) such that
$$
\bignorm{T\colon E(L^{\infty}(\Omega)\ten M) \longrightarrow
E(L^{\infty}(\Tdb)\ten M)} \leq C.
$$
For any $k=1,\ldots,n$ and any $x\in E(M)$, we have
$$
T(\varepsilon_k\otimes
x)=\sum_{A\subset\{1,\ldots,n\}}\Bigl(\int_{\Omega}\varepsilon_A\varepsilon_k\,
dm\,\Bigr)\,\eta_A\otimes x = \eta_k\otimes x.
$$ Hence $T$ maps $\sum_k\varepsilon_k\otimes x_k$ to $\sum_k\eta_k\otimes x_k$
for any $x_1,\ldots, x_n\in E(M)$ and we obtain that
$$
\Bignorm{\sum_k \eta_k\otimes x_k}_{E}\,\leq\, C \Bignorm{\sum_k
\varepsilon_k\otimes x_k}_{E}.
$$

A similar result holds with $\eta_k$ replaced by
$\widetilde{\eta_k}$. According to (\ref{2cos}), this yields the
estimate $\lesssim$ in the proposition. The reverse estimate is
proved similarly, using the fact (easy to check) that for any
$A\subset\{1,\ldots,n\}$ and any $k=1,\ldots, n$, we have
\begin{displaymath} \frac{1}{2\pi} \int_{0}^{2\pi} \eta_A(t) \eta_k(t)\, dt\, =
\left\{\begin{array}{ll} 0 & \hbox{if}\ A\not=\{k\},
\\ \frac{1}{2} & \hbox{if}\ A=\{k\}.
\end{array}\right.
\end{displaymath}
\end{proof}

Further equivalence properties of the above type are established
in \cite{AS}.

\begin{remark}\label{2Pisier}
The above result can be regarded as an analog of the following
classical result of Pisier. For any Banach space and for any
$1\leq r<\infty$, there is an equivalence
$$
\Bignorm{\sum_{k}\varepsilon_k\otimes x_k}_{L^r(\Omega:
X)}\,\approx\,\Bignorm{\sum_{k} e_{2^k}\otimes
x_k}_{L^r(\footnotesize{\Tdb};X)}
$$
for finite families $(x_k)_k$ of $X$ \cite{P1}.
\end{remark}

The main result of this section is the following noncommutative
Paley inequality, which extends \cite{LPP}.
\begin{theorem}\label{3Main}
Assume that $E$ is fully symmetric and that $q_E<\infty$.  There
is a constant $C\geq 0$ such that for any finite $(M,\tau)$, for
any $f\in H^{E}(L^{\infty}(\Tdb)\ten M)$ and for any $n\geq 1$, we
have
$$
\bignorm{\bigl(\widehat{f}(3^k)\bigr)_{k=1}^{n}}_{\inf} \, \leq\,
C\, \norm{f}_{E(L^{\infty}({\footnotesize{\Tdb}}) \ten M) }.
$$
\end{theorem}

\begin{proof} Throughout we let $(m_j)_{j\geq 0}$ be the sequence given by
$m_j=1$ if $j=3^k$ for some $k\geq 1$, and $m_j=0$ otherwise. Let
$f\in H^{E}(L^{\infty}(\Tdb)\ten M)$. Applying  Proposition
\ref{2Hardy1} (and Remark \ref{2Hardy2}) with
$N=L^\infty(\Tdb)\overline{\otimes}M$ and $\varepsilon =1$, we
find $g,h\in H^{E^{(2)}}(L^{\infty}(\Tdb)\ten M)$ such that $f=gh$
and
$$ \norm{g}_{E^{(2)}}\norm{h}_{E^{(2)}} \leq 2\norm{f}_E.
$$
By analyticity we have
$$
\widehat{f}(j) =\sum_{0\leq i\leq
j}\widehat{g}(i)\,\widehat{h}(j-i),\qquad j\geq 0.
$$
We define
$$
y_j=\sum_{\frac{j}{2}<i\leq j} \widehat{g}(i)\,
\widehat{h}(j-i)\qquad\hbox{and}\qquad z_j=\sum_{0\leq i\leq
\frac{j}{2}}\widehat{g}(i)\widehat{h}(j-i)
$$
for any $j\geq 0$, so that we have decompositions
$\widehat{f}(j)=y_j+z_j$ in $E(M)$. Thus it suffices to show that
for some constant $C\geq 0$, we have
\begin{equation}\label{3Estimate}
\bignorm{\bigl(m_j y_j\bigr)_{j=0}^{l}}_{c}\,\leq\,C\,\norm{f}_{E}
\qquad\hbox{and}\qquad
\bignorm{\bigl(m_jz_j\bigr)_{j=0}^{l}}_{r}\,\leq\,
C\,\norm{f}_{E}.
\end{equation}
for any $l\geq 1$. Given any integer $j\geq 0$, we define $$ g_j
=\,\sum_{\frac{j}{2}<i\leq j} e_i
\otimes\widehat{g}(i),\qquad\hbox{and}\qquad  h_j
=\,\sum_{\frac{j}{2}\leq i\leq j} e_i\otimes\widehat{h}(i)
$$
as elements of $H^{E^{(2)}}(L^{\infty}(\Tdb)\ten M)$. Then we have
$$
m_j y_j=m_j\,\sum_{i=0}^{j} \widehat{g_j}(i)\, \widehat{h}(j-i) =
\widehat{m_jg_j h}(j)
$$
and similarly,
$$
m_jz_j = \widehat{m_jg h_j}(j).
$$
We shall now concentrate on the first part of (\ref{3Estimate}).
Applying Lemma \ref{2Ampli}, we deduce from above that
$$
\bignorm{\bigl(m_j y_j\bigr)_{j=0}^{l}}_{c}
\,\leq\,C_E\,\Bignorm{\Bigl(\sum_{j=0}^{l} (m_j g_j h)^*(m_j g_j
h)\Bigr)^{\frac{1}{2}}}_{E(L^{\infty}(\footnotesize{\Tdb})\ten
M)}.
$$
Recall that for any $u$ and $v$ in some space of the form
$N+L^1(N)$, with $u\geq 0$, we have $(v^* u v)^{\frac{1}{2}} =
\vert u^{\frac{1}{2}} v\vert$. Consequently,
\begin{align*}
\Bignorm{\Bigl(\sum_{j=0}^{l}  (m_j g_j h)^*(m_j g_j
h)\Bigr)^{\frac{1}{2}}}_{E(L^{\infty}(\footnotesize{\Tdb})\ten M)}
& \,= \,\Bignorm{\Bigl(h^*\Bigl(\sum_{j=0}^{l} (m_j g_j)^* (m_j
g_j)
\Bigr)h\Bigr)^{\frac{1}{2}}}_{E(L^{\infty}(\footnotesize{\Tdb})\ten
M)}
\\ &\, =  \,\bignorm{\Bigl(\sum_{j=0}^{l}(m_j g_j)^* (m_j g_j)
\Bigr)^{\frac{1}{2}} \cdotp
h}_{E(L^{\infty}(\footnotesize{\Tdb})\ten M)}
\\&\, \leq \,\bignorm{\Bigl(\sum_{k=1}^{l} g_{3^k}^*
g_{3^k}\Bigr)^{\frac{1}{2}} \cdotp
h}_{E(L^{\infty}(\footnotesize{\Tdb})\ten M)}
\end{align*}
Applying the `Cauchy-Schwarz inequality' (\ref{2CS}) on
$L^{\infty}(\Tdb)\overline{\otimes}M$, we therefore obtain that
$$
\bignorm{\bigl(m_j
y_j\bigr)_{j=1}^{l}}_{c}\,\leq\,C_E\,\norm{h}_{E^{(2)}}\,
\Bignorm{\Bigl(\sum_{k=0}^{l} g_{3^k}^*g_{3^k}
\Bigr)^{\frac{1}{2}}}_{E^{(2)}}.
$$
For $l\geq 1$ as above, consider the linear map
$$
T\colon L^2(\Tdb; L^2(M))\longrightarrow L^2(M_l \otimes
L^{\infty}(\Tdb)\overline{\otimes} M)
$$
defined by $$ T(\varphi)=\left[\begin{array}{cccc} \varphi_1 & 0 &
\cdots & 0\\\vdots &\vdots & \ &\vdots\\   \vdots &\vdots & \
&\vdots\\\varphi_l & 0 & \cdots &
0\end{array}\right],\qquad\hbox{with}\quad
\varphi_k=\sum_{\frac{3^k}{2}<i\leq 3^k} e_i
\otimes\widehat{\varphi}(i).
$$
It is plain that $T$ is a contraction on $L^2$. Now let
$2<q<\infty$ and consider $\varphi\in L^q(\Tdb; L^q(M))$. The easy
part of the noncommutative Khintchine inequalities on $L^q(M)$
(which follows from the $2$-convexity of $L^q$) yields
$$
\norm{T(\varphi)}_{L^q(M_l \otimes
L^{\infty}(\footnotesize{\Tdb})\overline{\otimes} M)}\,=\,
\Bignorm{\Bigl(\sum_{k=1}^{l}
\varphi_k^*\varphi_k\Bigr)^{\frac{1}{2}}}_{L^q(\footnotesize{\Tdb};
L^q(M))}\,\leq \, \Bignorm{\sum_{k=1}^l\varepsilon_k\otimes
\varphi_k}_{L^q(\Omega\times \footnotesize{\Tdb}; L^q(M))}.
$$
Furthermore there exists  a constant $K_q$ (only depending on $q$)
such that
$$
\Bignorm{\sum_{k=1}^l
\theta_k\varphi_k}_{L^q(\footnotesize{\Tdb};L^q(M))}\,\leq\,K_q
\norm{\varphi}_{L^q(\footnotesize{\Tdb};L^q(M))}
$$
for any $\theta_k =\pm 1$ and any $\varphi\in L^q(\Tdb;L^q(M))$.
Indeed, $L^q(M)$ is a UMD Banach space, hence the latter estimate
follows from the vector-valued Fourier multiplier theory on this
class (\cite{MC}, see also \cite{B}). We deduce that
$$
\bignorm{T\colon L^q(\Tdb; L^q(M))\longrightarrow L^q(M_l \otimes
L^{\infty}(\Tdb) \overline{\otimes} M)))} \leq K_q.
$$
We now use interpolation. The space $E^{(2)}$ is 2-convex by
nature, and we know that $q_{E^{(2)}}=2q_E<\infty$. Hence by
\cite[Thm. 7.3]{KM}, there exists $2<q<\infty$ such that $E\in{\rm
Int}(L^2;L^q)$. Thus by Proposition \ref{2Interpolation1}, we have
$$
\bignorm{T\colon E^{(2)}(L^{\infty}(\Tdb)\ten M)\longrightarrow
E^{(2)}(M_l \otimes L^{\infty}(\Tdb)\ten  M)} \leq K
$$
for some constant $K$ not depending on $l\geq 1$. Now observe that
$T(g)=\sum_{k=1}^l E_{k1}\otimes g_{3^k}$. According to
(\ref{2Square}), this yields
$$
\Bignorm{\Bigl(\sum_{k=1}^{l} g_{3^k}^*
g_{3^k}\Bigr)^{\frac{1}{2}}}_{E^{(2)}}\,\leq\,K\,\norm{g}_{E^{(2)}}.
$$
Consequently, we have
$$
\bignorm{\bigl(m_j y_j\bigr)_{j=0}^{l}}_{c}\,\leq\,K C_E\,
\norm{g}_{E^{(2)}}\norm{h}_{E^{(2)}}\,\leq 2 KC_E \norm{f}_{E}.
$$
This is the first part of (\ref{3Estimate}), and the proof of the
second one is similar, using the $h_j$'s.
\end{proof}

\begin{proof} [{\it Proof of Theorem \ref{1Main}.}] We first prove
the lower estimate (part (1)). Let $x_1,\ldots,x_n$ in $E(M)$, and
let $s\in M$ be a selfadjoint projection with $\tau(s)<\infty$.
Consider the finite von Neumann algebra $sMs$ and the analytic
polynomial
$$
f_s=\sum_{k=1}^n e_{3^k}\otimes sx_ks.
$$
This is an element of $H^{E}(L^{\infty}(\Tdb)\ten sMs)$. Hence by
Theorem \ref{3Main}, we have an estimate
$$
\bignorm{(sx_ks)_k}_{\inf }\,\leq\, C  \norm{f_s}_E\leq
C\Bignorm{\sum_{k} e_{3^k}\otimes x_k}_{E}.
$$
Then applying Proposition \ref{2Duality} (2) and Proposition
\ref{2Sidon1}, we deduce the lower estimate (\ref{1K1}).

We now turn to the upper estimate (part (2)). We first assume that
$E$ is separable. For any $x_1,\ldots,x_n$ in $E(M)$ and
$y_1,\ldots, y_n$ in $E'(M)$, we have
$$
\Bigl\vert \sum_{k=1}^n\tau(x_k y_k)\Bigr\vert\,\leq
\,\bignorm{(x_k)_{k=1}^n}_{\max} \bignorm{(y_k)_{k=1}^n}_{\inf}.
$$
By assumption $q_E<\infty$ hence $p_{E'}>1$. Applying the first
part of Theorem \ref{1Main} to $E'$, we therefore  deduce that $$
\Bigl\vert \sum_{k=1}^n\tau(x_k y_k)\Bigr\vert\,\leq \,C
\bignorm{(x_k)_{k=1}^n}_{\max} \Bignorm{\sum_{k=1}^n
\varepsilon_k\otimes y_k}_{E'}
$$
for some constant $C\geq 0$ not depending either on $n$, the
$x_k$'s or the $y_k$'s. Now applying the first equivalence in
Remark \ref{2Kconvex} yields the result.

The proof in the case when $E$ is the dual of a separable
symmetric space is similar, using the second equivalence in Remark
\ref{2Kconvex}.
\end{proof}

\medskip\section{Examples}

In the noncommutative setting, it is well-known that the
quantities $\norm{\ }_{\max}$ and $\norm{\ }_{\inf}$ appearing in
Theorem \ref{1Main} are not equivalent in general. In Proposition
\ref{5Optimal} below we will show that the latter theorem cannot
be improved in general. For the time being we will consider
special cases when the estimates (\ref{1K1}) or  (\ref{1K2}) can
be replaced by an equivalence. Throughout we let $(M,\tau)$ be an
arbitrary semifinite von Neumann algebra, and we assume that $E$
is either separable or is the dual of a separable symmetric space.

\begin{corollary}\label{5Equiv1} \  \begin{itemize}
\item [(1)] Assume that $E\in {\rm Int}(L^1,L^2)$. Then
\begin{equation}\label{5Equiv11}
\Bignorm{\sum_k\varepsilon_k\otimes x_k}_{E}\,\approx\,
\bignorm{(x_k)_k}_{\inf}
\end{equation}
for finite families $(x_k)_k$ of $E(M)$. \item [(2)] Assume $E\in
{\rm Int}(L^2,L^q)$ for some $q<\infty$. Then
\begin{equation}\label{5Equiv12}
\Bignorm{\sum_k\varepsilon_k\otimes x_k}_{E}\,\approx\,
\bignorm{(x_k)_k}_{\max}
\end{equation}
for finite families $(x_k)_k$ of $E(M)$.
\end{itemize}
\end{corollary}

\begin{proof}
We only prove part (1), the proof of (2) being similar. By the
easy part of the noncommutative Khintchine inequalities on
$L^1(M)$ (equivalently, by the $2$-concavity of $L^1$), we have
$$
\Bignorm{\sum_k\varepsilon_k\otimes x_k}_{L^1(\footnotesize{\Tdb};
L^1(M))}\, \leq\,\Bignorm{\Bigl(\sum_{k} x_k^*
x_k\Bigr)^{\frac{1}{2}}}_{L^1(M)}.
$$
for any finite family $(x_k)_k$ in $L^1(M)$. On the other hand we
obviously  have
$$
\Bignorm{\sum_k\varepsilon_k\otimes x_k}_{L^2(\footnotesize{\Tdb};
L^2(M))}\, = \Bigl(\sum_{k}
\norm{x_k}_{L^2(M)}^2\Bigr)^{\frac{1}{2}}\, =
\,\Bignorm{\Bigl(\sum_{k} x_k^* x_k\Bigr)^{\frac{1}{2}}}_{L^2(M)}
$$
for $x_k\in L^2(M)$. Applying Lemma \ref{2Interpolation2}, we
deduce an estimate
$$
\Bignorm{\sum_k\varepsilon_k\otimes x_k}_{E}\,
\lesssim\,\bignorm{(x_k)_{k}}_{c}
$$
for $x_k\in E(M)$. The same holds true with $\norm{\ }_r$ instead
of $\norm{\ }_c$, and these two estimates immediately imply that
$\bignorm{\sum_k \varepsilon_k \otimes x_k}_{E}\,
\lesssim\,\bignorm{(x_k)_{k}}_{\inf}$.

Our assumption ensures that $q_E\leq 2$, see (\ref{2Encad}). The
converse inequality is therefore given by the lower estimate in
Theorem \ref{1Main}.
\end{proof}

The next results should be compared with \cite{LPX}.

\begin{corollary}\label{5Equiv3} \  \begin{itemize}
\item [(1)] If $E$ is $2$-concave or $q_E<2$, then the equivalence
property (\ref{5Equiv11}) holds true. \item [(2)] Assume that
$q_E<\infty$. If $E$ is $2$-convex or $p_E>2$, then the
equivalence property (\ref{5Equiv12}) holds true.
\end{itemize}
\end{corollary}

\begin{proof} This follows from the previous corollary.
Indeed by \cite[Thm. 7.3]{KM}, the assumption in (1) ensures that
$E\in{\rm Int}(L^1, L^2)$ whereas the assumption in (2) ensures
that $E\in{\rm Int}(L^2, L^q)$ for some $q<\infty$.
\end{proof}

\bigskip
It is well-known that the Rademacher averages
$\bignorm{\sum_k\varepsilon_k\otimes x_k}_{E}$ studied in the
present paper and the `classical' Rademacher averages
$\bignorm{\sum_k\varepsilon_k\otimes x_k}_{\Rad{E}}$ are not
equivalent in general. In fact, rather little is known on the
cases when
\begin{equation}\label{5G=T}
\Bignorm{\sum_k\varepsilon_k\otimes x_k}_{E}\,\approx\,
\Bignorm{\sum_k\varepsilon_k\otimes x_k}_{\Rad{E}},\qquad x_k\in
E(M).
\end{equation}
In the commutative setting, we have the following two
characterizations at our disposal. First it follows from
\cite[Prop. 2.d.1]{LT2} and \cite{AB} that $$
\Bignorm{\sum_k\varepsilon_k\otimes x_k}_{E}\,\approx\,
\Bignorm{\Bigl(\sum_k \vert x_k \vert^2 \Bigr)^{\frac{1}{2}}
}_{E(M)},\qquad x_k\in E(M),
$$
for all commutative $M$ if and only if $q_E<\infty$. Second, $$
\Bignorm{\sum_k\varepsilon_k\otimes
x_k}_{\Rad{E}}\,\approx\,\Bignorm{ \Bigl(\sum_k \vert
x_k\vert^2\Bigr)^{\frac{1}{2}}}_{E(M)},\qquad x_k\in E(M),
$$
for all commutative $M$ if and only if $E$ is $q$-concave for some
$q<\infty$ (see \cite[Cor. 1]{M} or \cite[Thm. 1.d.6 (i) and Thm.
1.f.12 (ii)]{LT2}).

We recall that if $E$ is $q$-concave, then $q_E\leq q$ (see e.g.
\cite[p. 132]{LT2}). Thus if $E$ is $q$-concave for some
$q<\infty$, (\ref{5G=T}) holds true when $M$ is commutative. On
the other hand,  consider $E=L^{r,\infty}$, with $1\leq r<\infty$.
Then $p_E=q_E =r$ but $E$ is $q$-concave for no finite $q$.
Consequently the equivalence (\ref{5G=T}) does not hold true in
general for that space $E$.

We do not know if (\ref{5G=T}) holds true for any $q$-concave $E$
(with $q<\infty$) and for any $M$. Combining Corollary
\ref{5Equiv3} and \cite{LPX}, we obtain classes of symmetric
spaces having this equivalence property.

\begin{corollary} Suppose  that either $E$ is 2-concave, or
$E$ is 2-convex and $q$-concave for some $q<\infty$. Then
(\ref{5G=T}) holds true for any $M$.
\end{corollary}

The paper \cite{DS} contains two classes of spaces $E$ satisfying
(\ref{5G=T}). On the one hand, it is shown that all reflexive
Orlicz spaces on $(0,\infty)$ have this property. On the other
hand, for any $1<p,q<\infty$ and any $\gamma\in\Rdb$, the
Lorentz-Zygmund space $L^{p,q}({\rm Log}L)^\gamma$ is also shown
to satisfy (\ref{5G=T}). All these spaces have non trivial Boyd
indices. Combining with Theorem \ref{1Main}, we derive that for
all such spaces $E$, we have estimates
\begin{equation}\label{5Genuine}
\bignorm{(x_k)_k}_{\inf}\,\lesssim\,\Bignorm{\sum_k\varepsilon_k\otimes
x_k}_{\Rad{E}}\,\lesssim\, \bignorm{(x_k)_k}_{\max},\qquad x_k\in
E(M).
\end{equation}

\bigskip We conclude this section by showing a certain optimality
of Theorem \ref{1Main} and of (\ref{5Genuine}). For any $p\geq 1$,
we let $S^p = L^p(B(\ell^2))$ denote the Schatten $p$-class on
$\ell^2$ and we let $S^p_m =  L^p(M_m)$ be its finite dimensional
version. It is well-known (and easy to check) that for any $p\not=
2$,
\begin{equation}\label{5Not}
\bignorm{(x_k)_k}_{\inf}\not\approx \bignorm{(x_k)_k}_{\max}\qquad
\hbox{on} \ S^p.
\end{equation}

\begin{proposition}\label{5Optimal} Let $E=L^p(0,\infty)\cap L^q(0,\infty)$,
with $1<p<2<q<\infty$.
\begin{itemize}
\item [(1)] $E$ satisfies the equivalence property (\ref{5G=T}).
\item [(2)] There exist a semifinite von Neumann algebra $M$ and
infinite dimensional subspaces $Y,Z\subset E(M)$ such that
$$
\Bignorm{\sum_{k} \varepsilon_k\otimes x_k}_{\Rad{E}}\,\approx\,
\bignorm{(x_k)_k}_{\max}\qquad\hbox{and}\qquad
\bignorm{(x_k)_k}_{\max}\not\approx \bignorm{(x_k)_k}_{\inf}
$$
for finite families $(x_k)_k$ of $Y$, whereas
$$
\Bignorm{\sum_{k}\varepsilon_k\otimes
x_k}_{\Rad{E}}\,\approx\,\bignorm{(x_k)_k}_{\inf}
\qquad\hbox{and}\qquad \bignorm{(x_k)_k}_{\inf}\not\approx
\bignorm{(x_k)_k}_{\max}
$$
for finite families $(x_k)_k$ of $Z$.
\end{itemize}
\end{proposition}

\begin{proof}
Note that $E(M)=L^p(M)\cap L^q(M)$. Then using the
Khintchine-Kahane inequality, we have for  a finite family
$(x_k)_k$ of $E(M)$
\begin{align*}
\Bignorm{\sum_k \varepsilon_k\otimes x_k}_{E(\Omega;E(M))}\, &
\approx\, \Bignorm{\sum_k \varepsilon_k\otimes x_k}_{L^p(\Omega;
L^p(M))} + \Bignorm{\sum_k \varepsilon_k\otimes x_k}_{L^q(\Omega;
L^q(M))}\\ & \approx \, \Bignorm{\sum_k \varepsilon_k\otimes
x_k}_{{\rm Rad}(L^p(M))} + \Bignorm{\sum_k \varepsilon_k\otimes
x_k}_{{\rm Rad}(L^q(M))}\\  & \approx \, \Bignorm{\sum_k
\varepsilon_k\otimes x_k}_{{\rm Rad}(L^p(M)\cap L^q(M))}.
\end{align*}
This proves (1).

To show (2), let $R$ be the hyperfinite $II_1$ factor. and let $$
M=R\mathop{\oplus}\limits^\infty B(\ell^2).
$$
Since the trace on $R$ is normalized, we have $L^q(R)\subset
L^p(R)$ with $\norm{\ }_p\leq \norm{\ }_q$ on $L^q(R)$. On the
other hand, $S^p\subset S^q$ with  $\norm{\ }_q\leq \norm{\ }_p$
on $S^p$.  Consequently, we have a topological direct sum
decomposition
$$
E(M) = L^q(R)\oplus S^p.
$$
Then take $Y=L^q(R)\oplus(0)$ and $Z=(0)\oplus S^p$.  For any
integer $m\geq 1$, there is a completely isometric embedding
$$
J_m\colon S^q_m\longrightarrow L^q(R),
$$
in the sense of \cite{P2} (see also \cite{P3}). In particular for
any $n\geq 1$ and any $x_1,\ldots, x_n\in S^q_m$, we have
$$
\bignorm{(J_m(x_k))_k}_c = \Bignorm{(I_{S^q_n}\otimes
J_m)\Bigl(\sum_{k=1}^n E_{k1}\otimes x_k\Bigr)}_{L^q(M_n(R))}
=\Bignorm{ \sum_{k=1}^n E_{k1}\otimes x_k}_{L^q(M_n\otimes M_m)} =
\bignorm{(x_k)_k}_c.
$$
The same holds true with $\norm{\ }_r$ instead of $\norm{\ }_c$
and we deduce that $$ \bignorm{(J_m(x_k))_k}_{\inf} =
\bignorm{(x_k)_k}_{\inf} \qquad\hbox{and}\qquad
\bignorm{(J_m(x_k))_k}_{\max} = \bignorm{(x_k)_k}_{\max}.
$$
By means of these equalities and (\ref{5Not}), we obtain that $Y$
and $Z$ have the properties stated in the proposition.
\end{proof}

\medskip\section{Double sums}
Let $E$ and $M$ as in Section 2. Let $(x_{ij})_{1\leq i,j\leq n}$
be a doubly indexed family of some $E(M)$. In this section we will
be interested in the double Rademacher average $\sum_{i,j}
\varepsilon_i\otimes \varepsilon_j\otimes x_{ij}$. Extending the
definitions of Section 2, we let $$ \Bignorm{\sum_{i,j=1}^n
\varepsilon_i\otimes \varepsilon_j\otimes x_{ij}}_{E} $$ denote
the norm of this sum in the noncommutative space
$E(L^{\infty}(\Omega)\ten L^{\infty}(\Omega)\ten M)$. The analysis
of this norm requires more definitions. We will use the matrix
notation $[x_{ij}]$ to denote the element $\sum_{i,j=1}^n
E_{ij}\otimes x_{ij}\,$ of $E(M_n(M))$. Accordingly we will write
$$
\bignorm{[x_{ij}]}_E\, =\, \Bignorm{\sum_{i,j=1}^n E_{ij}\otimes
x_{ij}}_{E(M_n(M))}
$$
for the norm of this matrix in $E(M_n(M))$. Also we extend the
notation (\ref{1Square}) to doubly indexed families by writing
$$
\bignorm{(x_{ij})_{i,j}}_c\,=\,\Bignorm{\Bigl(\sum_{i,j=1}^n
x_{ij}^* x_{ij}\Bigr)^{\frac{1}{2}}}_{E(M)}\qquad \hbox{and}\qquad
\bignorm{(x_{ij})_{i,j}}_r\,=\,\Bignorm{\Bigl(\sum_{i,j=1}^n
x_{ij}x_{ij}^* \Bigr)^{\frac{1}{2}}}_{E(M)}.
$$
Finally we introduce $\max$ and $\inf$ norms as follows. First we
let
$$
\bignorm{(x_{ij})_{i,j}}_{\max}\, =\,
\max\bigl\{\bignorm{[x_{ij}]}_E,\, \bignorm{[x_{ji}]}_E , \,
\bignorm{(x_{ij})_{i,j}}_c,\, \bignorm{(x_{ij})_{i,j}}_r\bigr\}.
$$
Second we let
$$
\bignorm{(x_{ij})_{i,j}}_{\inf}\, =\,
\inf\bigl\{\bignorm{[a_{ij}]}_E + \bignorm{[b_{ji}]}_E +
\bignorm{(c_{ij})_{i,j}}_c + \bignorm{(d_{ij})_{i,j}}_r\bigr\},
$$
where the infimum runs over all $4$-tuples $(a_{ij})_{i,j}$,
$(b_{ij})_{i,j}$, $(c_{ij})_{i,j}$, and $(d_{ij})_{i,j}$ of
families of $E(M)$ such that $x_{ij} = a_{ij} + b_{ij} + c_{ij} +
d_{ij}$ for any $1\leq i,j\leq n$.

Such norms were introduced in \cite{HP} and \cite{P2} on
noncommutative $L^p$-spaces and the following equivalence
properties hold. If $2\leq p<\infty$,
\begin{equation}\label{6DbK1} \Bignorm{\sum_{i,j}
\varepsilon_i\otimes \varepsilon_j\otimes
x_{ij}}_{L^p(\Omega\times\Omega; L^p(M))} \,\approx\,
\bignorm{(x_{ij})_{i,j}}_{\max},\qquad x_{ij}\in L^p(M),
\end{equation}
and if $1\leq p\leq 2$,
\begin{equation}\label{6DbK2}
\Bignorm{\sum_{i,j} \varepsilon_i\otimes \varepsilon_j\otimes
x_{ij}}_{L^p(\Omega\times\Omega; L^p(M))} \,\approx\,
\bignorm{(x_{ij})_{i,j}}_{\inf},\qquad x_{ij}\in L^p(M).
\end{equation}
See \cite[Sect. 3]{HP} and \cite[Rem. 9.8.9]{P2} for proofs and
more remarks. The main purpose of this section is to extend
(\ref{6DbK1}) (resp. (\ref{6DbK2})) to the case when $L^p$ is
replaced by a 2-convex space such that $q_E<\infty$ (resp. a
2-concave space $E$). We will repeatedly use the following simple
lemma.

\begin{lemma}\label{6lem} For any family $(x_{ij})_{1\leq i,j\leq n}$ in $E(M)$,
 we have
$$
\Bignorm{\sum_{i,j=1}^n E_{ij}\otimes x_{ij}}_{E(M_n\otimes M)}\,
=\, \Bignorm{\sum_{i,j=1}^n E_{i1}\otimes E_{1j}\otimes
x_{ij}}_{E(M_n\otimes M_n\otimes M)}.
$$
\end{lemma}

\begin{proof}
Indeed, the two elements $z_1=\sum_{i,j=1}^n E_{ij}\otimes x_{ij}$
and $z_2=\sum_{i,j=1}^n E_{i1}\otimes E_{1j}\otimes x_{ij}$ have
the same distribution function, hence $\mu(z_1)=\mu(z_2)$.
\end{proof}

We start with a general result and the 2-convex case.

\begin{proposition}\label{6Upper} Assume that $E$ is separable,
or that $E$ is the dual of a separable symmetric function space.
Assume further that $q_E<\infty$ and $p_E>1$. Then we have $$
\bignorm{(x_{ij})_{i,j}}_{\inf}\,\lesssim\, \Bignorm{\sum_{i,j}
\varepsilon_i\otimes \varepsilon_j\otimes x_{ij}}_{E}\, \lesssim\,
\bignorm{(x_{ij})_{i,j}}_{\max}
$$
for finite families $(x_{ij})_{i,j}$ of $E(M)$.
\end{proposition}

\begin{proof} The upper estimate is a simple reiteration argument.
For any $x_{ij}\in E(M)$, we write $$ \sum_{i,j}
\varepsilon_i\otimes \varepsilon_j\otimes x_{ij} = \sum_i
\varepsilon_i\otimes z_i,\qquad\hbox{with}\qquad\ z_i=\sum_j
\varepsilon_j\otimes x_{ij},
$$
and we apply the upper estimate of Theorem \ref{1Main} on
$L^{\infty}(\Omega)\ten M$. We find that \begin{align*}
\Bignorm{\sum_{i,j} \varepsilon_i \otimes \varepsilon_j\otimes
x_{ij}}_{E}\, & \lesssim\, \Bignorm{\sum_i E_{1i}\otimes z_i}_E\,
+\,\Bignorm{\sum_i E_{i1}\otimes z_i}_E \\ &  = \, \Bignorm{\sum_j
\varepsilon_j \otimes \Bigl(\sum_i E_{1i}\otimes
x_{ij}\Bigr)}_{E}\, +\, \Bignorm{\sum_j \varepsilon_j \otimes
\Bigl(\sum_i E_{i1}\otimes x_{ij}\Bigr)}_{E}.
\end{align*}
Applying Theorem \ref{1Main} (2) again, together with Lemma
\ref{6lem}, we see that
\begin{align*}
\Bignorm{\sum_j \varepsilon_j \otimes \Bigl(\sum_i E_{1i}\otimes
x_{ij}\Bigr)}_{E}\, & \lesssim\, \Bignorm{\sum_{ij} E_{1j} \otimes
E_{1i} \otimes x_{ij} }_E\, +\, \Bignorm{\sum_{ij} E_{j1} \otimes
E_{1i} \otimes x_{ij} }_E\\ &  = \, \bignorm{(x_{ij})_{i,j}}_r \,
+\, \bignorm{[x_{ji}]}_E.
\end{align*}
Likewise,
$$
\Bignorm{\sum_j \varepsilon_j \otimes \Bigl(\sum_i E_{i1}\otimes
x_{ij}\Bigr)}_{E} \,\lesssim\,  \bignorm{(x_{ij})_{i,j}}_c \, +\,
\bignorm{[x_{ij}]}_E.
$$
Combined with the previous inequality, these yield the desired
upper estimate.

The lower estimate can be deduced from the upper one by duality,
the argument being similar to the one in the proof of Theorem
\ref{1Main} given in Section 3. We skip the details.
\end{proof}

\begin{theorem}\label{6Double1}  Assume that $E$ is separable, or
that $E$ is the dual of a separable symmetric function space. If
$E\in {\rm Int}(L^2,L^q)$ for some $q<\infty$, then we have
$$
\Bignorm{\sum_{i,j} \varepsilon_i\otimes \varepsilon_j\otimes
x_{ij}}_{E} \,\approx\, \bignorm{(x_{ij})_{i,j}}_{\max}
$$
for finite families $(x_{ij})_{ij}$ of $E(M)$.

In particular, this holds true if $q_E<\infty$ and if either $E$
is 2-convex or $p_E>2$.
\end{theorem}

\begin{proof}
Assume that  $E\in {\rm Int}(L^2,L^q)$ for some $q<\infty$. The
estimate $\lesssim$ is given by Proposition \ref{6Upper}. As in
the proof of Corollary \ref{5Equiv1}, the reverse estimate is
proved by interpolation, using (\ref{6DbK1}) on $L^q$.

The last line of the statement then follows from \cite[Thm.
7.3]{KM}.
\end{proof}

\bigskip
We shall now consider double sums in the 2-concave case or, more
generally, in the case when $E\in {\rm Int}(L^1,L^2)$. This case
turns out to be much more delicate than the 2-convex one. The
major obstacle is that we do not know whether the lower estimate
in Proposition \ref{6Upper} remains true in the case when $1\leq
p_E\leq q_E<\infty$.

We will need to somehow replace the Rademacher functions by the
generators of a free group living in the associated group von
Neumann algebra. The use of such techniques goes back to Haagerup
and Pisier \cite{HP}. In the sequel, we let $G=\Fdb_\infty$ be a
free group with an infinite sequence $\gamma_1,\ldots, \gamma_n,
\ldots$ of generators. Let $(\delta_g)_{g\in G}$ denote the
canonical basis of $\ell^2_G$ and let $\lambda\colon G\to
B(\ell^2_G)$ be the left regular representation of $G$, defined by
$$ \lambda(g)\delta_h = \delta_{gh},\qquad g,h\in G.
$$
We recall that the group von Neumann algebra of $G$ is defined as
$$
VN(G)=\bigl\{\lambda(g)\, :\, g\in G\bigr\}'' = \overline{{\rm
Span}}^{w^*}\bigl\{\lambda(g)\, :\, g\in G\bigr\} \subset
B(\ell^2_G).
$$
For simplicity we let $\M=VN(G)$ in the sequel. Let $e$ be the
unit element of $G$. Then $\M$ has a canonical normalized  normal
trace $\sigma$ defined by $\sigma(z)=\langle
z(\delta_e),\delta_e\rangle$.

Let $(M,\tau)$ be an arbitrary semifinite von Neumann algebra. In
the sequel we regard the von Neumann tensor product
$\M\overline{\otimes}M$ as equipped with $\sigma\otimes\tau$ in
the usual way. It is remarkable that the noncommutative Khintchine
inequalities on $L^p$ remain unchanged if one replaces the
Rademacher sequence by the $\lambda(\gamma_k)$'s. Namely for any
$1\leq p<\infty$, there is an equivalence
\begin{equation}\label{4Equiv}
\Bignorm{\sum_{k} \lambda(\gamma_k) \otimes
x_k}_{L^p(\footnotesize{\M}\ten M)} \,\approx\, \Bignorm{\sum_{k}
\varepsilon_k \otimes x_k}_{L^p(\footnotesize{\Tdb};  L^p(M))}
\end{equation}
for finite families $(x_k)_k$ of $L^p(M)$ (see \cite[Sect. 3]{HP}
and \cite[Thm. 9.8.7]{P2}).

Let $x_1,\ldots, x_n\in E(M)$. It is clear that each
$\lambda(\gamma_k)\otimes x_k$ belongs to
$E(\M\overline{\otimes}M)$ and we write
$$
\Bignorm{\sum_{k}\lambda(\gamma_k) \otimes x_k}_{E}
$$
for the norm of their sum $\sum_k  \lambda(\gamma_k) \otimes
x_k\,$ in the latter space.

Let $a\in \M$ and let $\varphi_a\colon L^1(\M)\to\Cdb$ be the
functional defined by $\varphi_a (z)=\sigma(za)$ for any $z\in
L^1(\M)$. The algebraic tensor product $L^1(\M)\otimes L^1(M)$ is
dense in $L^1(\M\overline{\otimes}M)$ (see e.g. \cite[Sect.
3]{ER}) and $\varphi_a\otimes I_{L^1(M)}$ uniquely extends to a
bounded operator $T_a^1\colon L^1(\M\overline{\otimes}M)\to
L^1(M)$. Indeed, this extension is the pre-adjoint of the
embedding $M\to \M\overline{\otimes}M$ taking any $x\in M$ to
$a\otimes x$. Likewise, $\varphi_{a \vert \footnotesize{\M}}
\otimes I_{M}$ uniquely extends to a contractive normal operator
$T_a^\infty \colon \M\overline{\otimes}M\to M$. Moreover these two
maps coincide on the intersection of $L^1(\M\overline{\otimes}M)$
and $\M\overline{\otimes}M$. Hence we may define a bounded linear
map
$$
T_a\colon L^1(\M\overline{\otimes}M) +\M\overline{\otimes}M
\longrightarrow L^1(M) + M $$ extending both of them. In the
sequel it will be convenient to write
$$
\langle u,a\rangle =T_a(u),\qquad u\in L^1(\M\overline{\otimes}M)
+\M\overline{\otimes}M.
$$

Fix an integer $n\geq 1$ and let
$$
P_n\colon L^1(\M\overline{\otimes}M) +\M\overline{\otimes}M
\longrightarrow L^1(\M\overline{\otimes}M) +\M\overline{\otimes}M
$$
be defined by
\begin{equation}\label{4P0}
P_n(u)=\,\sum_{k=1}^n \lambda(\gamma_k) \otimes \langle
u,\lambda(\gamma_k^{-1})\rangle\,.
\end{equation}
This is a projection which extends the orthogonal projection
$L^2(\M\overline{\otimes}M)\to L^2(\M\overline{\otimes}M)$ onto
the subspace ${\rm Span}\{\lambda(\gamma_k) :\, 1\leq k\leq
n\}\otimes L^2(M)$.

\begin{lemma}\label{4P3}
There exist a constant $K_E \geq 0$ such that for any $n\geq 1$
and any $(M,\tau)$ as above,
$$
\bignorm{P_n\colon E(\M\ten M)\longrightarrow E(\M\ten M)}\leq
K_E.
$$
\end{lemma}

\begin{proof}
By construction, $P_n$ is the extension of two bounded maps
$$
P_n^\infty \colon \M\overline{\otimes}M\longrightarrow
\M\overline{\otimes}M\qquad\hbox{and} \qquad P_n^1\colon
L^1(\M\overline{\otimes}M) \longrightarrow
L^1(\M\overline{\otimes}M).
$$
Since the $L^2$-realization of $P_n$ is selfadjoint, $P_n^\infty$
is the adjoint of the mapping $v\mapsto [P_n^1(v^*)]^*$ on
$L^1(\M\overline{\otimes}M)$. Hence
\begin{equation}\label{4Equal}
\norm{P_n^\infty}=\norm{P_n^1}.
\end{equation}
Let $\P= {\rm Span}\{\lambda(g) :\, g\in G\}$. According to
\cite[Prop. 1.1]{HP}, the restriction of $P_n^\infty$ to $\P
\otimes M$ has norm $\leq 2$. Further, $\P\otimes M$ is a
$w^*$-dense $*$-subalgebra of  $\M\overline{\otimes}M$. Hence the
unit ball of $\P \otimes M$ is $w^*$-dense in the unit ball of
$\M\overline{\otimes}M$ by Kaplansky's Theorem. Since $P_n^\infty$
is $w^*$-continuous, we deduce that $\norm{P_n^\infty}\leq 2$. The
result now follows from (\ref{4Equal}) and Proposition
\ref{2Interpolation1}.
\end{proof}

\begin{lemma}\label{4lem}
Assume that $q_E<\infty$. Then we have an estimate
$$
\Bignorm{\sum_{k}\varepsilon_k \otimes x_k}_{E} \, \lesssim\,
\Bignorm{\sum_{k}  \lambda(\gamma_k) \otimes x_k}_{E}
$$
for finite families $(x_k)_k$ of $E(M)$.
\end{lemma}

\begin{proof}
For any $n\geq 1$, let us define
$$
Q_n\colon L^1(\M\overline{\otimes}M) +\M\overline{\otimes}M
\longrightarrow L^1(\Tdb; L^1(M)) +L^{\infty}(\Tdb)
\overline{\otimes}M
$$
in a similar way to $P_n$, by letting
$$
Q_n(u)=\,\sum_{k=1}^n \varepsilon_k \otimes \langle
u,\lambda(\gamma_k^{-1})\rangle\,.
$$
Applying Lemma \ref{4P3} with $E=L^p$ and (\ref{4Equiv}), we
obtain that for any $1\leq p<\infty$, we have
$$
\bignorm{Q_n\colon L^p(\M\overline{\otimes}M) \longrightarrow
L^p(\Tdb; L^p(M))}\leq K_p
$$
for some constant $K_p$ only depending on $p$. Since $q_E<\infty$,
there exists some $1<q<\infty$ such that $E\in{\rm Int}(L^1,L^q)$.
Applying the above estimate with $p=1$ and $p=q$ together with
Proposition \ref{2Interpolation1}, we deduce that
$$
\bignorm{Q_n\colon E(\M\overline{\otimes}M) \longrightarrow
E(L^{\infty}(\Tdb)\ten M)}\leq D_E
$$
for some constant $D_E$ only depending on $E$. Since $$
Q_n\Bigl(\sum_{k=1}^n \lambda(\gamma_k)\otimes x_k\Bigr)\,=\,
\sum_{k=1}^n \varepsilon_k\otimes x_k
$$
for any $x_1,\ldots, x_n$ in $E(M)$, we obtain the desired
estimate.
\end{proof}

Using $K$-convexity as in Remark \ref{2Kconvex}, it is not hard to
see that if $p_E>1$ and $q_E<\infty$, then the two averages
$\bignorm{\sum_{k} \lambda(\gamma_k) \otimes x_k}_{E}$ and
$\bignorm{\sum_{k} \varepsilon_k  \otimes x_k}_{E}$ are actually
equivalent. We do not know if this equivalence holds true if we
merely assume that  $q_E<\infty$. However we have the following
special case.

\begin{lemma}\label{5Equiv2}
Assume that $E\in {\rm Int}(L^1,L^2)$. Then we have an equivalence
$$
\Bignorm{\sum_{k} \lambda(\gamma_k) \otimes x_k}_{E} \,\approx\,
\Bignorm{\sum_{k} \varepsilon_k  \otimes
x_k}_{E}\,\Bigl(\,\approx\, \norm{(x_k)_k}_{\inf}\,\Bigr)
$$
for finite families $(x_k)_k$ of $E(M)$.
\end{lemma}

\begin{proof}
Indeed arguing as in the proof of Corollary \ref{5Equiv1} (1) and
using equivalence (\ref{4Equiv}), we see that
$$
\Bignorm{\sum_{k} \lambda(\gamma_k) \otimes x_k}_{E}
\,\lesssim\,\bignorm{(x_k)_k}_{\inf}.
$$
Combining with Lemma \ref{4lem} and the lower estimate in Theorem
\ref{1Main}, one gets the equivalence.
\end{proof}

\begin{theorem}\label{6Double2}
Assume that $E$ is separable, or that $E$ is the dual of a
separable symmetric function space. If $E\in {\rm Int}(L^1,L^2)$,
then we have
$$
\Bignorm{\sum_{i,j} \varepsilon_i\otimes \varepsilon_j\otimes
x_{ij}}_{E} \,\approx\, \bignorm{(x_{ij})_{i,j}}_{\inf}
$$
for finite families $(x_{ij})_{ij}$ of $E(M)$.

In particular, this holds true if  either $E$ is 2-concave or
$q_E<2$.
\end{theorem}

\begin{proof} By \cite[Thm. 7.3]{KM}, the last assertion will follow from the main one.
Thus we assume that $E\in {\rm Int}(L^1,L^2)$. Then the estimate
$$ \Bignorm{\sum_{i,j=1}^n \varepsilon_i \otimes \varepsilon_j
\otimes x_{ij}}_{E} \,\lesssim\, \bignorm{(x_{ij})_{i,j}}_{\inf}
$$
follows from interpolation principles as in Corollary
\ref{5Equiv1} (1), using (\ref{6DbK2}) for $p=1$. We are now going
to concentrate on the converse inequality.

We first observe that
\begin{equation}\label{6Equiv}
\Bignorm{\sum_{i,j} \lambda(\gamma_i) \otimes \lambda(\gamma_j)
\otimes x_{ij}}_{E} \,\approx\, \Bignorm{\sum_{i,j} \varepsilon_i
\otimes \varepsilon_j \otimes x_{ij}}_{E}
\end{equation}
for finite doubly indexed families $(x_{ij})_{i,j}$ of $E(M)$,
where  $\norm{\cdots}_{E}$ in the left hand-side stands for the
norm of the double sum $\sum_{i,j} \lambda(\gamma_i) \otimes
\lambda(\gamma_j) \otimes x_{ij}$ in the space $E(\M\ten\M\ten
M)$. Indeed applying Lemma \ref{5Equiv2} first on
$L^\infty(\Omega)\ten M$ and then on $\M\ten M$, we have
\begin{align*}
\Bignorm{\sum_{i} \varepsilon_i \otimes \Bigl( \sum_j
\varepsilon_j \otimes x_{ij}\Bigr)}_{E} & \approx
\Bignorm{\sum_{i} \lambda(\gamma_i) \otimes \Bigl( \sum_j
\varepsilon_j \otimes x_{ij}\Bigr) }_{E}\\ & = \Bignorm{\sum_{j}
\varepsilon_j \otimes \Bigl(\sum_i \lambda(\gamma_i)  \otimes
x_{ij}\Bigr) }_{E}\\ & \approx \Bignorm{\sum_{j} \lambda(\gamma_j)
\otimes \Bigl(\sum_i \lambda(\gamma_i)  \otimes x_{ij}\Bigr)
}_{E},
\end{align*}
which yields (\ref{6Equiv}). It therefore suffices to show an
estimate
\begin{equation}\label{6Goal}
\Bignorm{\sum_{i,j} \lambda(\gamma_i) \otimes \lambda(\gamma_j)
\otimes x_{ij}}_{E}\,\gtrsim \bignorm{(x_{ij})_{i,j}}_{\inf}.
\end{equation}
Fix an integer $n\geq 1$ and let $(x_{ij})_{1\leq i,j\leq n}$ be a
family of $E(M)$. We write
$$
\sum_{i,j=1}^n \lambda(\gamma_i) \otimes \lambda(\gamma_j) \otimes
x_{ij} = \sum_{i=1}^n \lambda(\gamma_i) \otimes
z_i,\qquad\hbox{with}\qquad\ z_i=\sum_{j=1}^n  \lambda(\gamma_j)
\otimes x_{ij}.
$$
Then according to Lemma \ref{5Equiv2}, there exist $z'_1,\ldots,
z'_n,\, z''_1,\ldots, z''_n$ in $E(\M\ten M)$ such that
$z_i=z'_i+z''_i$ for any $i$ and $$ \Bignorm{\sum_{i=1}^n E_{i1}
\otimes z'_i}_{E(M_n\otimes\footnotesize{\M}\ten M)}\,\lesssim\,
\Bignorm{\sum_{i,j} \lambda(\gamma_i) \otimes \lambda(\gamma_j)
\otimes x_{ij}}_{E},
$$
$$
\Bignorm{\sum_{i=1}^n E_{1i} \otimes
z''_i}_{E(M_n\otimes\footnotesize{\M}\ten M)}\,\lesssim\,
\Bignorm{\sum_{i,j} \lambda(\gamma_i) \otimes \lambda(\gamma_j)
\otimes x_{ij}}_{E}.
$$
Let $P_n$ be defined by (\ref{4P0}). We clearly have
$P_n(z_i)=z_i$, hence $z_i = P_n(z'_i) + P_n(z''_i)$ for any
$i=1,\ldots,n$. Moreover there exist two families $(u_{ij})_{1\leq
i,j\leq n}$ and $(v_{ij})_{1\leq i,j\leq n}$ of $E(M)$ such that
$$
P_n(z'_i)=\sum_{j=1}^n \lambda(\gamma_j)\otimes
u_{ij}\qquad\hbox{and}\qquad P_n(z''_i)=\sum_{j=1}^n
\lambda(\gamma_j)\otimes v_{ij}
$$
for any $i$. Then we have decompositions $$
x_{ij}=u_{ij}+v_{ij},\qquad 1\leq i,j\leq n.
$$
Let us now apply Lemma \ref{4P3} on $M_n\otimes M$, with
$I_{M_n}\otimes P_n$ instead of $P_n$. We find that \begin{align*}
\Bignorm{\sum_{j=1}^n \lambda(\gamma_j)\otimes\Bigl(\sum_{i=1}^n
E_{i1} \otimes  u_{ij} \Bigr)}_{E}\, & = \Bignorm{\sum_{i=1}^n
E_{i1} \otimes P_n(z'_i)}_{E}\\ & \lesssim\, \Bignorm{\sum_{i=1}^n
E_{i1} \otimes  z'_i}_{E}\\ & \lesssim \, \Bignorm{\sum_{i,j}
\lambda(\gamma_i) \otimes \lambda(\gamma_j) \otimes x_{ij}}_{E}.
\end{align*}
In the same manner,
$$
\Bignorm{\sum_{j=1}^n\lambda(\gamma_j)\otimes\Bigl(\sum_{i=1}^n
E_{1i}\otimes v_{ij} \Bigr)}_{E}\, \lesssim \, \Bignorm{\sum_{i,j}
\lambda(\gamma_i) \otimes \lambda(\gamma_j) \otimes x_{ij}}_{E}.
$$
We can now repeat the above arguments, using the projection  ${\rm
Col}_n$ and (\ref{2Col}) (as well as its row counterpart) instead
of $P_n$. We therefore obtain new families $(a_{ij})_{i,j}$,
$(b_{ij})_{i,j}$, $(c_{ij})_{i,j}$, and $(d_{ij})_{i,j}$ in $E(M)$
such that
$$
u_{ij} = c_{ij}+ a_{ij},\qquad v_{ij} = b_{ij}+ d_{ij}, \qquad
1\leq i,j\leq n,
$$
and
$$
\Bignorm{\sum_{i,j} E_{j1}\otimes E_{i1}\otimes
c_{ij}}_E\,\lesssim\,
\Bignorm{\sum_j\lambda(\gamma_j)\otimes\Bigl(\sum_i E_{i1}\otimes
u_{ij} \Bigr)}_{E},
$$
$$
\Bignorm{\sum_{i,j} E_{1j}\otimes E_{i1}\otimes
a_{ij}}_E\,\lesssim\, \Bignorm{\sum_j
\lambda(\gamma_j)\otimes\Bigl(\sum_i E_{i1}\otimes u_{ij}
\Bigr)}_{E},
$$
$$
\Bignorm{\sum_{i,j} E_{j 1}\otimes E_{1i}\otimes
b_{ij}}_E\,\lesssim\,
\Bignorm{\sum_j\lambda(\gamma_j)\otimes\Bigl(\sum_i E_{1i}\otimes
v_{ij} \Bigr)}_{E},
$$
$$
\Bignorm{\sum_{i,j} E_{1j}\otimes E_{1i}\otimes
d_{ij}}_E\,\lesssim\, \Bignorm{\sum_j
\lambda(\gamma_j)\otimes\Bigl(\sum_i E_{1i}\otimes  v_{ij}
\Bigr)}_{E}.
$$
According to Lemma \ref{6lem}, these estimate imply that the four
quantities $$ \bignorm{[a_{ij}]}_E,\ \bignorm{[b_{ji}]}_E,\,
\Bignorm{\Bigl(\sum_{ij} c_{ij}^* c_{ij}\Bigr)^{\frac{1}{2}}}_E,\
\hbox{ and } \Bignorm{\Bigr(\sum_{ij} d_{ij}  d_{ij}^*
\Bigr)^{\frac{1}{2}}}_E
$$
are all $$ \lesssim \Bignorm{\sum_{i,j} \lambda(\gamma_i)
\otimes\lambda(\lambda_j) \otimes x_{ij}}_E.
$$
Furthermore we have $x_{ij}= u_{ij}+v_{ij} = a_{ij}+ b_{ij}+
c_{ij}+ d_{ij}$ for any $i,j$, hence we obtain that (\ref{6Goal})
holds true.
\end{proof}

\end{document}